\documentclass[11pt, reqno]{amsart}
\usepackage[margin=3.6cm]{geometry}

\usepackage{amsmath, amsfonts, amsthm, amssymb, dsfont}
\usepackage{exscale}
\usepackage{epsfig}
\usepackage{amscd}
\usepackage{graphics}
\usepackage{color}

\usepackage{lipsum}

\renewcommand{\geq}{\geqslant}
\renewcommand{\leq}{\leqslant}

\newtheorem{proposition}{Proposition}[section]
\newtheorem{corollary}[proposition]{Corollary}
\newtheorem{lemma}[proposition]{Lemma}
\newtheorem*{main-theorem}{Main Theorem}
\newtheorem*{theorem*}{Theorem}
\newtheorem{theorem}{Theorem}

\theoremstyle{definition}

\newtheorem{remark}[proposition]{Remark}

\newtheorem*{remark*}{Remark}

\numberwithin{equation}{section}

\def\phi{\varphi}

\def\reals{{\mathbb R}}

\def\phi{\varphi}

\def\be{\begin{eqnarray*}}
\def\ee{\end{eqnarray*}}
\def\ben{\begin{eqnarray}}
\def\een{\end{eqnarray}}

\def\L2R{L_{\text{Rest}}^2}

\def\11{\mathds{1}}

\def\L2c{L^2_{\text{comp}}}

\usepackage{dirtytalk}

\usepackage{url}

\newcommand\numberthis{\addtocounter{equation}{1}\tag{\theequation}}

\begin{document}

\title[Observability for the Schr\"{o}dinger Equation on Simplices]{Asymptotic Boundary Observability for the Schr\"{o}dinger Equation on Simplices}
\author{Sarah Carpenter and Hans Christianson}

\begin{abstract}
    We consider the Schr\"{o}dinger equation $(i\partial_t+\Delta)u=0$ on an $n$-dimensional simplex with Dirichlet boundary conditions. We use a commutator argument along with integration by parts to obtain an observability asymptotic for any one face of the simplex. Rather than the typical observability inequality, we are able to do better as we instead prove a large-time asymptotic. Note that this paper parallels \cite{lu}, in which Christianson-Lu prove the analogous result with the wave equation. 
\end{abstract}

\maketitle
%%%%%%%%%%%%%%%%%%%%%%%%%%%%%%%%%%%%%%%%%%%%%%%%%%%%%%%%%%%
\section{Introduction}
In this paper, we study boundary observability for solutions to the
Schr\"{o}dinger equation on $n$-dimensional simplices. The result is a
large-time asymptotic observability identity on any one face. This
paper is part of a collection by the first author and collaborators
considering boundary observability for the wave equation and
equidistribution of Neumann data mass on a triangle or simplex. 
In \cite{equidistribution}, it is shown that the $L^2$ norm of the (semi-classical) Neumann data on each side is \textit{equal} to the length of the side divided by the area of the triangle, and a generalization to simplices is given in \cite{equidistribution-n}. An asymptotic boundary observability for solutions to the wave equation is proved for triangles in \cite{stafford} and for simplices in \cite{lu}, and we  prove this asymptotic for solutions to the Schr\"{o}dinger equation on triangles and simplices in this paper.

The proofs are similar to these other papers, in which we use
a commutator argument and integration by parts, while the proof for
simplices will also require linear algebra and  symplectic
geometry. Note that this is a much simpler approach than the
traditional controllability/observability argument that uses geometric
optics and microlocal analysis.  However, the proof is particular to
simplices and does not work for other polytopes.  In fact, the main
result is false in general.

Let $\Omega \subset \mathbb{R}^n$ be a non-degenerate simplex with faces $G_0,\dots,G_n$. Let $\mbox{Vol}_n(\Omega)$ be the volume of $\Omega$ and $\mbox{Vol}_{n-1}(G_j)$ be the $(n-1)$-dimensional induced volume of $G_j$. 
We consider the Schr\"{o}dinger equation with Dirichlet boundary conditions on $\Omega$:
\begin{equation} \label{SEn}
    \begin{cases} 
      (i\partial_t+\Delta)u=0 \\
      u|_{\partial \Omega}=0 \\
      u(x,0)=u_0(x), x\in \mathbb{R}^n, \\
   \end{cases}
\end{equation}
where $u_0 \in H_0^1(\Omega) \cap H^s(\Omega)$ $\forall \, s\geq 0$. Next we consider the (conserved) energy for the Schr\"{o}dinger equation, defined by the $\dot{H}^1$ mass,
\begin{equation*}
    E(t)=\int_\Omega |\nabla u|^2 \,dx.
\end{equation*}
\begin{theorem}
  \label{T:main}
Suppose $u$ solves the Schr\"{o}dinger equation (\ref{SEn}). Then $\forall \, T> 0$, the Neumann data on each of the boundary faces satisfies 
\begin{equation}
    \int_0^T\int_{G_j} |\partial_\omega u|^2 dS_j \,dt=\frac{2T \, {\mbox{Vol}_{n-1}(G_j)}}{n \, \mbox{Vol}_n(\Omega)}E(0)\left(1+\mathcal{O}\left(\frac{1}{T}\right)\right),
\end{equation}
where $\partial_\omega$ is the normal derivative on $\partial\Omega$, and $dS_j$ is the surface measure on $G_j$. 
\end{theorem}

\begin{remark}
The assumption that $u_0 \in H_0^1(\Omega) \cap H^s(\Omega)$ for all
$s \in \reals$ is overkill.  We just make this assumption so we can
integrate by parts without worrying about regularity issues.

\end{remark}

Using the Poincar\'e inequality, we know there exists $C>0$ such that
$\| u \| \leq C \| \nabla u \|$, which gives the following Corollary.

\begin{corollary}
\label{C:main}
  Under the assumptions of Theorem \ref{T:main}, there exists a constant
$c>0$ such that
  \begin{equation}
    \label{E:poincare-obs}
\int_0^T\int_{G_j} |\partial_\omega u|^2 dS_j \,dt \geq c \| u_0
\|^2_{L^2(\Omega)}
\end{equation}
for each $j = 0 , 1 , \ldots n$.

\end{corollary}

\begin{remark}
The statement of Corollary \ref{C:main} is the more familiar
observability {\it inequality} rather than the asymptotic in Theorem
\ref{T:main}.  The estimate \eqref{E:poincare-obs} says one can
``observe'' the initial $L^2$ norm by taking a measurement on one side
of the simplex.  It is very interesting to note that we have an
asymptotic observation of $\| \nabla u_0 \|^2$ and observation
inequality of $\| u_0 \|^2$.

\end{remark}

\subsection{History}
A landmark result of controllability was \cite{rt}, in which Rauch and
Taylor showed exponential decay of the energy of solutions to damped hyperbolic equations in bounded domains given the geometric control condition is satisfied - that is every ray hits the region of control in some finite time. Rauch and Taylor consider the control region being both a fixed subregion of the domain as well as a fixed subset of the boundary.
The closely related idea of observability for solutions of the wave
equation  observes the initial energy  by taking a measurement in the
control region. Another landmark result is that of Bardos-Lebeau-Rauch
\cite{blr}, in which they prove a similar condition for the boundary
in that every ray must hit the observability region on the boundary
transversally.

These results make heavy use of microlocal analysis and geometric
optics.  To get an idea of the subtlety to these proofs,
the papers of Lebeau \cite{Leb-damp}, Christianson \cite{Chr-NC,Chr-NC-erratum}, and
Burq-Christianson \cite{BuCh-dw} show that if the geometric control condition
fails in a weak sense, then there is a sharp loss in energy decay rate
and regularity.  One of the novelties of \cite{stafford,lu} and the present work
is that it does not require a geometric control assumption.

Now these results are not applicable for solutions to the Schr\"{o}dinger equation as it is not hyperbolic. Controllability and observability have certainly been studied with the Schr\"{o}dinger equation, but majorly on interior subsets of the domain as the observability region. Jaffard \cite{jaffard} proved an internal control for solutions to the Schr\"{o}dinger equation, which was extended by Burq, Zworski, and Bourgain to control results on tori \cite{bz12,bbz,bz17}.
Lebeau \cite{lebeau} did, however, consider controllability on the boundary for subsets that satisfy the geometric control condition from \cite{blr}. 

\section{Proof for Planar Triangles}
In this section, we summarize the proof of Theorem \ref{T:main} for
triangles.  The proof for triangles does not require any special
change of variables as in the proof for simplices, so is a friendly
introduction to the main ideas.

Let $\Omega \subset \mathbb{R}^2$ be a triangle with sides $A,B,C$. Let $\ell_A, \ell_B, \ell_C$ denote the respective altitudes (the perpendicular distance from the side to the non-adjacent corner). Let $L$ be the length of the longest side. We consider the following initial/boundary value problem for the Schr\"{o}dinger equation: 
\begin{equation} \label{SE2}
    \begin{cases} 
      (i\partial_t+\Delta)u=0 \\
      u|_{\partial \Omega}=0 \\
      u(x,y,0)=u_0(x,y),
   \end{cases}
\end{equation}
where $u_0\in H_0^1(\Omega) \cap H^s(\Omega)$ $\forall \, s\geq 0$. We denote the (conserved) initial energy by 
\begin{equation*}
    E(t)=\int_\Omega |\nabla u|^2 \,dV.
\end{equation*} 

\begin{theorem}
  \label{T:main-tri}
Suppose $u$ solves the Schr\"{o}dinger equation (\ref{SE2}). Then $\forall \, T>0$, the Neumann data on side A satisfies 
\begin{equation}\label{EQdim2}
\int_0^T\int_A |\partial_\nu u|^2 dS \,dt=\frac{2T}{\ell_A}E(0)\left(1+\mathcal{O}\left(\frac{1}{T}\right)\right),
\end{equation}
where $\partial_\nu$ is the normal derivative on $\partial \Omega$ and $dS$ is the arc length measure. 
The analogous asymptotic on sides $B$ and $C$ also holds. 
\end{theorem}

\begin{remark}
As noted in \cite{stafford}, when dealing with solutions to the wave
equation, the appearance of the factor $1/\ell_A$ in (\ref{EQdim2}) is
due to the finite propagation speed, as it takes approximately time
$\ell_A$ for a wave to travel from the opposite corner to side
$A$. However, we do not see the infinite speed of propagation for
solutions of the
Schr\"{o}dinger equation in the present result.  
\end{remark}

As in \cite{stafford}, the proof is broken down into two cases: acute
and obtuse (or right) triangles. We  only prove the acute case as the obtuse case is very similar. Additionally, we show that this result does not hold generally on polygons by giving a counterexample on a square. 

Without loss of generality, we prove Theorem \ref{T:main-tri} only for side $A$, and thus we  let $\ell=\ell_A$. Let $\Omega$ be an acute triangle, oriented in such a way that side $A$ is parallel to the $y$-axis, and the corner opposite from $A$ is at the origin, as shown in Figure 1. Thus the altitude of length $\ell$ corresponds with the $x$-axis. We label the remaining sides $B$ and $C$ as in Figure 1. Let $a_1$ be the length of the part of $A$ below the $x$-axis, and $a_2$ the part above. 
\begin{figure}
  \centering \includegraphics[width=3in]{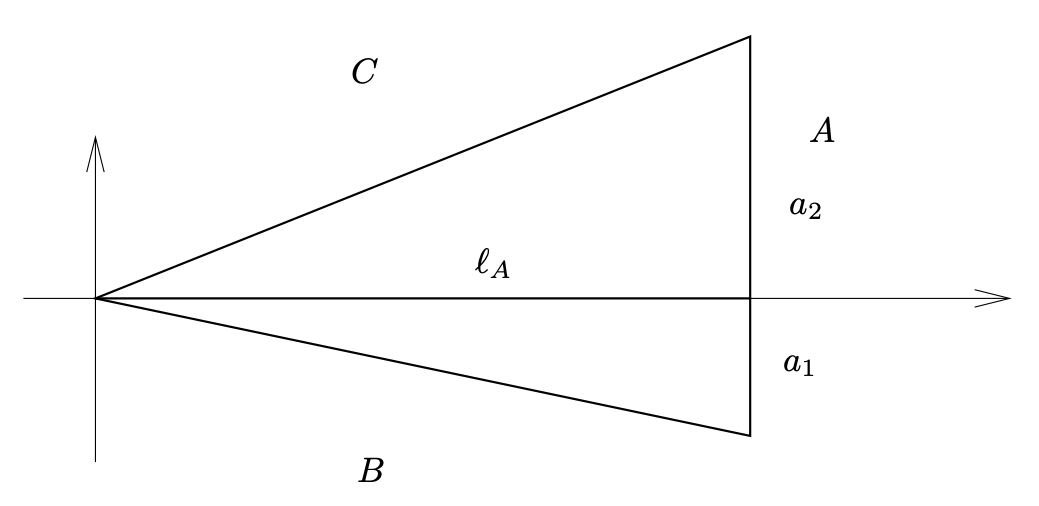}
  \caption{Setup for acute triangles.}
\end{figure}

\begin{proof}[Proof of Theorem \ref{T:main-tri}]
We begin by showing the energy is conserved. We use Green's Theorem and the fact that $u$ satisfies (\ref{SE2}).

\begin{align*}
    \frac{d}{\,dt} E(t) &= \frac{d}{\,dt} \int_\Omega \nabla u\cdot \nabla \bar{u} \,dV \\
    &= \int_\Omega 2\mbox{Re}\left(\nabla u_t\cdot \nabla \bar{u}  \right)\,dV \\
    &= 2\mbox{Re}\left(i\int_\Omega \nabla (\Delta u) \cdot \nabla \bar{u}\,dV  \right) \\  
    &= 2\mbox{Re}\left(-i\int_\Omega \Delta u \Delta \bar{u}\,dV +i\int_{\partial\Omega} \partial_\nu\bar{u} \Delta u \,dS \right) \\ 
    &= 2\mbox{Re}\left(-i\int_\Omega |\Delta u|^2 \,dV +\int_{\partial\Omega} (\partial_\nu\bar{u}) u_t \,dS \right) \\ 
    &= 2\mbox{Re}\left(-i\int_\Omega |\Delta u|^2 \,dV  \right) \\ 
    &=0
\end{align*}
Consider the vector field $X=x\partial_x+y\partial_y$ on $\Omega$ and the commutator $[i\partial_t+\Delta,X]$. Then 
\begin{align*}
    [i\partial_t+\Delta,X] &= (i\partial_t+\Delta)X-X(i\partial_t+\Delta) \\
%    &=\Delta X - X\Delta \\
    &=\Delta (x\partial_x)+\Delta (y\partial_y) - X\Delta  \\
%    &=\partial_x(\partial_x+x\partial_x^2)+x\partial_x\partial_y^2+ y\partial_y\partial_x^2+\partial_y(\partial_y+y\partial_y^2)- X\Delta  \\
%    &= 2\partial_x^2+x\partial_x^3+x\partial_x\partial_y^2+ y\partial_y\partial_x^2+2\partial_y^2+y\partial_y^3- X\Delta  \\
    &=2\Delta  +x\partial_x\Delta +y\partial_y\Delta  - X\Delta \\
    &=2\Delta.
\end{align*}
Therefore, using Green's Theorem on $u$ satisfying (\ref{SE2}),
\begin{align*}
    \int_0^T\int_\Omega ([i\partial_t+\Delta,X]u)\bar{u} \,dV \,dt 
    &= 2\int_0^T\int_\Omega( \Delta u)\bar{u} \,dV \,dt  \\
    &=-2\int_0^T\int_\Omega \nabla u \cdot \nabla \bar{u}  \,dV \,dt  +2\int_0^T \int_{\partial\Omega} (\partial_\nu u) \bar{u} \,dS \,dt   \\
    &=-2\int_0^T E(0) \,dt \\
    &= -2TE(0). \numberthis \label{2com1}
\end{align*}
Also note that if $u$ satisfies the Schr\"{o}dinger equation (\ref{SE2}), then $\overline{i u_t}+\Delta\bar{u}=0$, and also that $\overline{i u_t}=-i\bar{u}_t$. Now evaluating the integral directly using integration by parts/Green's Theorem and the fact that $u$ satisfies (\ref{SE2}),
\begin{align*}
    \int_0^T\int_\Omega ([i\partial_t+\Delta,X]u)\bar{u} \,dV \,dt 
    &=\int_0^T\int_\Omega (
    (i\partial_t+\Delta)Xu)\bar{u}-(X(i\partial_t+\Delta)u) \bar{u} \,dV \,dt \\
    &=\int_0^T\int_\Omega ( i\partial_t Xu)\bar{u} + (\Delta Xu)\bar{u} \,dV \,dt \\
    &= -\int_0^T\int_\Omega (iXu) \bar{u}_t + (\nabla Xu)\cdot (\nabla \bar{u})\,dV \,dt \\
    &\qquad \qquad + \int_\Omega (iXu) \bar{u}|_0^T \,dV + \int_0^T \int_{\partial\Omega} (\partial_\nu Xu)\bar{u} \,dS \,dt \\
    &=\int_0^T\int_\Omega (Xu)\overline{iu_t} + (Xu)\Delta \bar{u} \,dV \,dt \\
    &\qquad \qquad + \int_\Omega (iXu)\bar{u}|_0^T \,dV -  \int_0^T \int_{\partial\Omega} (Xu)(\partial_\nu \bar{u}) \,dS \,dt \\
    &=\int_\Omega (iXu)\bar{u}|_0^T \,dV -  \int_0^T \int_{\partial\Omega} (Xu)(\partial_\nu \bar{u}) \,dS \,dt. \numberthis \label{2com2}
\end{align*}
Combining our results from (\ref{2com1}) and (\ref{2com2}), we obtain 
\begin{equation}\label{EQ2}
    \int_0^T\int_{\partial\Omega} (Xu)(\partial_\nu \bar{u}) \,dS \,dt=2TE(0)+\int_\Omega (iXu)\bar{u}|_0^T \,dV.
\end{equation}
We now put a bound on the last term of (\ref{EQ2}) in terms of
the energy $E(0)$ using a Poincar\'e type inequality. The following lemma is stated and proved in \cite{stafford}. 
\begin{lemma}\label{lemma}
Let $u \in H_0^1(\Omega) \cap H^s(\Omega)$ $\forall \,s$. Then the following holds: 
\begin{equation*}
   ||u||_{L^2(\Omega)} \leq L\sqrt{e-1}||\partial_x u||_{L^2(\Omega)}
\end{equation*}
\end{lemma}
Trivially from this lemma we get that $||u||_{L^2(\Omega)}\leq L\sqrt{e-1}||\nabla u||_{L^2(\Omega)}$.
By the triangle inequality, Cauchy's inequality with parameter $\alpha>0$, and Lemma \ref{lemma},
\begin{align*}
    \left| \int_\Omega iXu\bar{u} \,dV  \right| &= \left| \int_\Omega i(x\partial_x+y\partial_y)u\bar{u} \,dV  \right| \\
    &\leq \int_\Omega |x\partial_xu||\bar{u}|+ |y\partial_yu||\bar{u}| \,dV \\
    &\leq \int_\Omega \alpha|u|^2+\frac{L^2}{2\alpha}(|\partial_xu|^2+|\partial_yu|^2) \,dV \\
    &\leq \int_\Omega \alpha L^2(e-1)|\nabla u|^2+\frac{L^2}{2\alpha}|\nabla u|^2 \,dV  \\
    &=\frac{3}{2}\int_\Omega L^2\sqrt{e-1}|\nabla u|^2 \,dV \\
    &\leq 2L^2\sqrt{e-1}E(0),
\end{align*}
upon taking $\alpha=\frac{1}{\sqrt{e-1}}$. Therefore (assuming a
non-zero solution), 
\begin{align*}
    \frac{1}{2TE(0)}\left| \int_0^T\int_{\partial\Omega} (Xu)(\partial_\nu \bar{u}) \,dS \,dt-2TE(0)\right| &=\frac{1}{2TE(0)}\left| \int_\Omega iXu\bar{u}|_0^T \,dV \right| \\
    &\leq \frac{L^2\sqrt{e-1}}{T},
\end{align*}
giving us 
\begin{equation}\label{Asym}
    \int_0^T\int_{\partial\Omega} (Xu)(\partial_\nu \bar{u}) \,dS \,dt=2TE(0)\left(1+\mathcal{O}\left(\frac{1}{T}\right)\right).
\end{equation}
We now obtain the Neumann data on $A$ from the left hand side of (\ref{Asym}). As $u|_{\partial\Omega}=0$, the tangential derivative along the boundary vanishes. On side $A$, the tangential derivative is $\partial_y$ and the normal derivative is $\partial_x$, so $X=x\partial_x+y\partial_y=x\partial_\nu$. Thus 
\begin{align*}
   \int_0^T\int_{A} (Xu)(\partial_\nu \bar{u}) \,dS \,dt 
   &= \int_0^T\int_{A} (x\partial_\nu u)(\partial_\nu \bar{u}) \,dS \,dt \\
   &=\int_0^T\int_{A} \ell|\partial_\nu u|^2 \,dS \,dt. \numberthis\label{intA}
\end{align*}
Now on side $B$, $y=-\frac{a_1}{\ell}x$, so the unit tangent vector is $\tau=(\frac{\ell}{b},-\frac{a_1}{b})$ and the unit normal vector is $\nu=(-\frac{a_1}{b},-\frac{\ell}{b})$. As the tangential derivative vanishes, 
\begin{equation*}
    \tau\cdot\nabla u=\frac{\ell}{b}\partial_x u-\frac{a_1}{b}\partial_y u =0,
\end{equation*}
so $\partial_y u=\frac{\ell}{a_1} \partial_x u$. Then 
\begin{equation*}
Xu =x\partial_x u+y\frac{\ell}{a_1} \partial_x u 
=x\partial_x u-x\frac{a_1}{\ell}\frac{\ell}{a_1}\partial_x u=0,
\end{equation*}
and therefore, 
\begin{equation}\label{intB}
    \int_0^T\int_{B} (Xu)(\partial_\nu \bar{u}) \,dS \,dt=0.
\end{equation}
Lastly, on side $C$, $y=\frac{a_2}{\ell}x$, so the unit tangent vector is $\tau=(-\frac{\ell}{c},-\frac{a_2}{c})$ and the unit normal vector is $\nu=(-\frac{a_1}{b},-\frac{\ell}{b})$. As the tangential derivative vanishes, 
\begin{equation*}
    \tau\cdot\nabla u=-\frac{\ell}{c}\partial_x u-\frac{a_2}{c}\partial_y u=0,
\end{equation*}
and therefore, $\partial_y u=-\frac{\ell}{a_2}\partial_x u$. Again we see that 
\begin{equation*}
    Xu=x\partial_x u-y\frac{\ell}{a_2} \partial_x u 
=x\partial_x u -x\frac{a_2}{\ell}\frac{\ell}{a_2} \partial_x u=0,
\end{equation*}
and thus 
\begin{equation}\label{intC}
    \int_0^T\int_{C} (Xu)(\partial_\nu \bar{u}) \,dS \,dt=0.
\end{equation}
Combining (\ref{intA}), (\ref{intB}), and (\ref{intC}) we obtain
\begin{align*}
\int_0^T\int_{A} \ell|\partial_\nu u|^2 \,dS \,dt &=
\int_0^T\int_{\partial\Omega} (Xu)(\partial_\nu \bar{u}) \,dS \,dt \\
&=2TE(0)\left(1+\mathcal{O}\left(\frac{1}{T}\right)\right).
\end{align*}
Finally, dividing by $\ell$, we obtain (\ref{EQdim2}).
\end{proof}

\subsection*{Failure on a Square Domain}
We  now show that this result does not hold generally on polygons by
giving a counterexample on a square. Consider the square domain
$\Omega=[0,2\pi]\times [0,2\pi]$ and the function
$u(x,y,t)=\frac{1}{\pi}e^{-it(1+n^2)}\sin(x)\sin(ny)$ for some integer
$n>0$. Then $u$ satisfies
\[ \begin{cases} 
      (i\partial_t+\Delta)u=0 \\
      u|_{\partial \Omega}=0 \\
      u(x,y,0)=u_0(x,y) 
   \end{cases}
\]
for $u_0\in H_0^1(\Omega) \cap H^s(\Omega)$ $\forall \, s\geq 0$ and has energy 
\begin{align*}
    E(0) &= \frac{1}{\pi^2} \int_0^{2\pi} \int_0^{2\pi} \cos^2(x)\sin^2(ny)+n^2\sin^2(x)\cos^2(ny) \,dx \,dy \\
    &= \frac{1}{\pi^2} \int_0^{2\pi} \int_0^{2\pi} \frac{1}{2}(1+\cos(2x))\sin^2(ny)+\frac{n^2}{2}(1-\cos(2x))\cos^2(ny) \,dx \,dy \\
    &=\frac{1}{\pi^2} \int_0^{2\pi} \pi \sin^2(ny)+n^2\pi \cos^2(ny) \,dy  \\
    &=\frac{1}{\pi} \int_0^{2\pi} \frac{1}{2}(1+\cos(2ny))+\frac{n^2}{2}(1-\cos(2ny)) \,dy \\
    &= 1+n^2. \numberthis\label{counterenergy}
\end{align*}
We show that along the right edge $\{2\pi\}\times [0,2\pi]$ the desired observability does not hold if $n$ is large enough. On this edge, $\partial_\nu=\partial_x$, so 
\begin{align*}
    \int_0^T\int_0^{2\pi} |\partial_\nu u|^2 |_{x=2\pi} \,dy \,dt 
    &= \int_0^T\int_0^{2\pi} \frac{1}{\pi^2} |e^{-it(1+n^2)}\cos(2\pi)\sin(ny)|^2 \,dy \,dt \\
    &= \frac{1}{\pi^2} \int_0^T\int_0^{2\pi} \sin^2(ny) \,dy \,dt \\
    &=\frac{T}{\pi} \\
    & = \frac{T}{\pi (1 + n^2)} E(0).
\end{align*}
Then there is no $C_T$ such that the observability condition 
\begin{equation*}
    \int_0^T\int_0^{2\pi} |\partial_\nu u|^2 |_{x=2\pi} \,dy \,dt \geq C_T E(0)%% =C_T(1+n^2)
\end{equation*}
holds for all solutions $u$. %% if we pick a large enough $n$, where the last equality comes from (\ref{counterenergy}).
Thus $u$ serves as a counterexample to  Theorem \ref{T:main-tri} on a polygon that is different from a triangle. 

\section{Proof of Theorem \ref{T:main}}

\begin{remark}
In the triangle proof, we were able to find an explicit constant from
Lemma \ref{lemma} to give us our asymptotic, however, it is not
sharp. In the proof of simplices, we will not find an explicit constant. 
\end{remark}

Let \{$p_1,\dots, p_n$\} $\subset \mathbb{R}^n$ be linearly independent vectors, and let $p_0$ denote the origin in $\mathbb{R}^n$. Then the $n$-dimensional simplex spanned by  \{$p_1,\dots, p_n$\} is defined by
\begin{equation}
    \Omega=\left\{ \sum_{j=0}^{n} t_jp_j : \sum t_j=1 \mbox{ and } t_j \geq 0  \right\}.
\end{equation}
The standard simplex is the simplex in which $p_j=e_j$ for each $j=1,\dots, n$, where $e_j$ are the standard basis vectors, and we denote it by $\widetilde{\Omega}$. We define the matrix 
$$ A=
\begin{bmatrix}
| & | & \cdots & | \\ p_1 & p_2 & \cdots & p_n \\ | & | & \cdots & |
\end{bmatrix},
$$ which is invertible as $p_1,\dots, p_n$ are linearly independent, and thus we let $B=A^{-1}$. Then for $x\in \mathbb{R}^n$, we let $y=Bx$. Then as $Bp_j=e_j$, we see that 
\begin{equation}
    \widetilde{\Omega}=\left\{ \sum_{j=0}^{n} t_jBp_j : \sum t_j=1 \mbox{ and } t_j \geq 0  \right\},
\end{equation}
and thus this transformation takes our simplex $\Omega$ to the standard simplex $\widetilde{\Omega}$.

Now we lift this transformation to $T^{*}\mathbb{R}^n$. For $\xi \in \mathbb{R}^n$ in our $x$-coordinates, we use symplectic geometry to see that $\eta=(B^{-1})^T\xi$ is the momentum variable in the $y$-coordinates. Now as the symbol of the Laplacian in the $x$-coordinates is $\xi^T\xi=\xi_1^2+\cdots+\xi_n^2$, the symbol of the Laplacian in the $y$-coordinates is $\xi^T\xi=\eta^TBB^T\eta$. We let $\Gamma=BB^T$. Thus the Laplacian in the $y$-coordinates is $$\widetilde{\Delta}=\sum_{i,j}^n \Gamma_{ij}\partial_{y_i}\partial_{y_j},$$ and we see that $-\widetilde{\Delta}$ is elliptic as $\Gamma$ is positive definite. Lastly, the energy in terms of the $y$-coordinates is given by 
\begin{equation}
    \widetilde{E}(t)=\int_{\widetilde{\Omega}} |B^T\nabla u|^2 \,dy,
\end{equation}
where $\nabla=\nabla_y$. Note that for the remainder of the paper, $\nabla$ will represent $\nabla_y$.
\begin{remark}
We only prove the Theorem \ref{T:main} on the side $G_0$ as we could begin by translating the simplex so that a different \say{corner} is the origin.
\end{remark}

\subsubsection*{Proof:}
We begin by proving that the energy is conserved on the standard simplex. We use the version of Green's formula from 
\cite{equidistribution}, and the fact that $u$ satisfies 
\begin{equation} \label{SEn0}
    \begin{cases} 
      (i\partial_t+\widetilde{\Delta})u=0 \\
      u|_{\partial \widetilde{\Omega}}=0. \\
   \end{cases}
\end{equation}
Also note that if $u$ satisfies (\ref{SEn0}), then $\overline{i u_t}+\widetilde{\Delta}\bar{u}=0$, and also that $\overline{i u_t}=-i\bar{u}_t$. Then
\begin{align*}
  \frac{d}{dt} E(t) &= \frac{d}{dt} \int_{\widetilde{\Omega}} (B^T\nabla u) \cdot (B^T\nabla \bar{u}) \,dy \\
  &=2\mbox{Re}\left(\int_{\widetilde{\Omega}} (B^T\nabla u_t) \cdot (B^T\nabla \bar{u}) \,dy \right) \\
  &=2\mbox{Re}\left(-\int_{\widetilde{\Omega}} (\nabla^TBB^T\nabla \bar{u})u_t \,dy +\int_{\partial\widetilde{\Omega}} (\nu^T BB^T \nabla \bar{u})u_t \,d\widetilde{S} \right) \\
    &=2\mbox{Re}\left(-\int_{\widetilde{\Omega}} (\widetilde{\Delta} \bar{u})u_t \,dy \right) \\
    &=2\mbox{Re}\left(-i\int_{\widetilde{\Omega}} (\bar{u}_t)u_t \,dy \right) \\ 
    &=2\mbox{Re}\left(-i\int_{\widetilde{\Omega}} |u_t|^2 \,dy \right) \\ 
    &=0,
\end{align*}
where $d\widetilde{S}$ is the surface measure and $\nu$ is the outward normal vector on the standard simplex $\widetilde{\Omega}$. 
Now consider the vector field $Y=y_1\partial_{y_1}+\cdots+y_n\partial_{y_n}$ and the commutator $[i\partial_t+\widetilde{\Delta},Y]$ on the standard simplex $\widetilde{\Omega}$. As $\widetilde{\Delta}$ is a constant coefficient symmetric operator, we have that $[i\partial_t+\widetilde{\Delta},Y]=2\widetilde{\Delta}$. 
%\begin{align*}
%    [i\partial_t+\widetilde{\Delta},Y]
%    &=\left[\sum_{i,j}^n \Gamma_{ij}\partial_{y_i}\partial_{y_j},\sum_{k=1}^n y_k\partial_{y_k}\right]\\
%    &=\sum_{i,k}^n \left[\Gamma_{ii}\partial_{y_i}^2, y_k\partial_{y_k}\right] +\sum_{k=1}^n\sum_{i\neq j}^n \left[\Gamma_{ij}\partial_{y_i}\partial_{y_j},y_k\partial_{y_k}\right] \\
%    &=\sum_{i=1}^n \left[\Gamma_{ii}\partial_{y_i}^2, y_i\partial_{y_i}\right] + \sum_{i\neq j}^n \left[\Gamma_{ij}\partial_{y_i}\partial_{y_j},y_i\partial_{y_i}+y_j\partial_{y_j}\right]  \\
%    &= \sum_{i=1}^n 2\Gamma_{ii}\partial_{y_i}^2 + \sum_{i\neq j}^n [\Gamma_{ij}\partial_{y_j}(\partial_{y_i}+y_i\partial_{y_i}^2)-y_i\Gamma_{ij}\partial_{y_i}^2\partial_{y_j} \\
%    &\qquad\qquad \qquad\qquad \qquad \qquad+\Gamma_{ij}\partial_{y_i}(\partial_{y_j}+y_j\partial_{y_j}^2)-y_j\Gamma_{ij}\partial_{y_j}^2\partial_{y_i}] \\
 %   &=\sum_{i=1}^n 2\Gamma_{ii}\partial_{y_i}^2 + \sum_{i\neq j}^n 2\Gamma_{ij}\partial_{y_i}\partial_{y_j} \\
%    &=2\widetilde{\Delta}
%\end{align*}
Using this along with Green's Theorem and the fact that $u$ satisfies (\ref{SEn0}), 
\begin{align*}
   \int_0^T \int_{\widetilde{\Omega}} [i\partial_t+\widetilde{\Delta},Y] u\bar{u} \,dy \,dt &= 2\int_0^T \int_{\widetilde{\Omega}}  (\widetilde{\Delta} u)\bar{u} \,dy \,dt \numberthis \label{lap} \\
    &=2\int_0^T \int_{\widetilde{\Omega}} (\nabla^TBB^T\nabla u)\bar{u} \,dy \,dt \\
   &=-2\int_0^T \int_{\widetilde{\Omega}} (B^T\nabla u) \cdot (B^T\nabla \bar{u}) \,dy \,dt \\
    &\qquad \qquad+\int_0^T\int_{\partial\widetilde{\Omega}} (\nu^T BB^T \nabla \bar{u})u_t \,d\widetilde{S} \,dt \\
    &=-2\int_0^T \int_{\widetilde{\Omega}}|B^T\nabla u|^2 \,dy \,dt  \\
    &=-2\int_0^T \widetilde{E}(0) \,dt \\
    &=-2T\widetilde{E}(0) \numberthis \label{com1}. 
\end{align*}
Now we compute the same integral without first simplifying the commutator. Again, we use integration by parts/Green's Theorem and the fact that $u$ satisfies (\ref{SEn0}). 
\begin{align*}
    \int_0^T \int_{\widetilde{\Omega}} ([i\partial_t+\widetilde{\Delta},Y] u)\bar{u} \,dy \,dt
    &=\int_0^T \int_{\widetilde{\Omega}}(
    (i\partial_t+\widetilde{\Delta})Y u) \bar{u}-(Y(i\partial_t+\widetilde{\Delta})u)\bar{u} \,dy \,dt \\
    &=\int_0^T \int_{\widetilde{\Omega}} i(\partial_tYu)\bar{u}+(\widetilde{\Delta}Y u)\bar{u} \,dy \,dt \\
    &=\int_0^T \int_{\widetilde{\Omega}} -i(Yu)\bar{u}_t+(Y u)(\widetilde{\Delta}\bar{u}) \,dy \,dt +\int_{\widetilde{\Omega}} i(Yu)\bar{u}|_0^T \,dy \\
    &\qquad \qquad +\int_0^T \int_{\partial\widetilde{\Omega}} (\nu^T BB^T \nabla(Yu))\bar{u} \,d\widetilde{S} \,dt \\ &\qquad \qquad -\int_0^T \int_{\partial\widetilde{\Omega}} (Yu)(\nu^T BB^T \nabla\bar{u}) \,d\widetilde{S} \,dt \\
    &= \int_0^T \int_{\widetilde{\Omega}} (Yu)(\overline{i  u_t}+\widetilde{\Delta}\bar{u}) \,dy \,dt +\int_{\widetilde{\Omega}} i(Yu)\bar{u}|_0^T \,dy \\
    &\qquad \qquad -\int_0^T \int_{\partial\widetilde{\Omega}} (Yu)(\nu^T BB^T \nabla\bar{u}) \,d\widetilde{S} \,dt \\
    &=\int_{\widetilde{\Omega}} i(Yu)\bar{u}|_0^T \,dy -\int_0^T \int_{\partial\widetilde{\Omega}} (Yu)(\nu^T BB^T \nabla\bar{u}) \,d\widetilde{S} \,dt. \numberthis \label{com2}
\end{align*}
Combining our results from (\ref{com1}) and (\ref{com2}), we obtain \begin{equation}\label{EQ}
     \int_0^T \int_{\partial\widetilde{\Omega}} (Yu)(\nu^T BB^T \nabla\bar{u}) \,d\widetilde{S} \,dt=2T\widetilde{E}(0)+\int_{\widetilde{\Omega}} iYu\bar{u}|_0^T \,dy. 
\end{equation}
Now the point of our change of coordinates was to make the computation of the normal vectors easier. On our standard simplex $\widetilde{\Omega}$, we denote the boundary faces $F_0, \dots, F_n$, where $F_j$ denotes the face in which $y_j=0$ for $j=1,\dots,n$, and $F_0$ the remaining face. Thus the normal vector on $F_j$ is $$\nu_j=-e_j$$ for $j=1,\dots,n$. Then on $F_0$, the normal vector is $$\nu_0=n^{-1/2}(1,\dots,1).$$
Thus the normal derivative for $j=1,\dots,n$ is $$\partial_{\nu_j}=-\partial_{y_j},$$ and $$ \partial_{\nu_0}=n^{-1/2}(\partial_{y_1}+\cdots+\partial_{y_n}). $$ By our Dirichlet boundary conditions, the tangential derivatives of $u$ vanish. Thus on $F_j$ for $j=1,\dots,n$, $\partial_\ell u=0$  except for $\ell=j$. But as $y_j=0$ on $F_j$, we see that for $j=1,\dots,n$, $$Yu|_{F_j}=y_1\partial_{y_1}u+\cdots+y_n\partial_{y_n}u=0.$$
As $(1,-1,0,\dots,0)$ is tangent to $F_0$, we see that $\partial_{y_1}u=\partial_{y_2}u$, and similarly, we see that $$\partial_{y_1}u=\cdots=\partial_{y_n}u.$$
Thus for $j=1,\dots,n$, 
\begin{equation} \label{normal}
    \partial_{\nu_0}u=n^{-1/2}(\partial_{y_1}u+\cdots+\partial_{y_n}u)=n^{-1/2}(n\partial_{y_j}u)=n^{1/2} \partial_{y_j}u.
\end{equation}
Then as $y_1+\cdots+y_n=1$ on $F_0$, 
$$ Yu|_{F_0}=y_1\partial_{y_1}u+\cdots+y_n\partial_{y_n}u=(y_1+\cdots+y_n)n^{-1/2}\partial_{\nu_0}u=n^{-1/2}\partial_{\nu_0}u. $$
Now we rewrite (\ref{EQ}) as 
\begin{equation}\label{reEQ}
     \int_0^T \int_{F_0} (n^{-1/2}\partial_{\nu_0}u)(\nu_0^T BB^T \nabla\bar{u}) \,d\widetilde{S}_0 \,dt=2T\widetilde{E}(0)+\int_{\widetilde{\Omega}} iYu\bar{u}|_0^T \,dy. 
\end{equation}

Our goal now is to bound the last term of (\ref{reEQ}) in terms of the energy $E(0)$. Note that the following constant $C$ changes with each calculation. Indeed, using the triangle inequality and Cauchy's inequality (on the first and fourth step),
\begin{align*}
    \left| \int_{\widetilde{\Omega}} iYu\bar{u} \,dy \right|
    &\leq \int_{\widetilde{\Omega}} C\left( \sum_{j=1}^n |y_j\partial_{y_j}u| \right)^2+C|\bar{u}|^2 \,dy \\
    &\leq \int_{\widetilde{\Omega}} C\left( \sum_{j=1}^n |\partial_{y_j}u| \right)^2+C|u|^2 \,dy \\
    &= \int_{\widetilde{\Omega}} C \sum_{i,j}^n (|\partial_{y_i}u||\partial_{y_j}u|) +C|u|^2 \,dy \\
    &\leq \int_{\widetilde{\Omega}} C\sum_{i,j}^n (|\partial_{y_i}u|^2+|\partial_{y_j}u|^2) +C|u|^2 \,dy \\
    &\leq \int_{\widetilde{\Omega}} C|\nabla u|^2 +C|u|^2 \,dy \\
    &\leq \int_{\widetilde{\Omega}} C|\nabla u|^2 \,dy,
\end{align*}
where the last step follows from the Poincaré inequality. 

Now this is almost the energy term that we are looking for, but recall that the energy on $\widetilde{\Omega}$ has the transformation $B$ in it. As $-\widetilde{\Delta}$ is an elliptic operator, $||\nabla u||_{L^2} \leq C \langle -\widetilde{\Delta}u,u\rangle_{L^2}$ for some $C$. Using this along with our calculations in (\ref{lap}),
\begin{align*}
    \left| \int_{\widetilde{\Omega}} iYu\bar{u} \,dy \right|
    &\leq \int_{\widetilde{\Omega}} C(-\widetilde{\Delta}u)\bar{u} \,dy \\
    &=C\widetilde{E}(0).
\end{align*}
Therefore, we see that (for non-zero solutions) 
\begin{align*}
    \frac{1}{2T\widetilde{E}(0)} \left|  \int_0^T \int_{F_0} (n^{-1/2}\partial_{\nu_0}u)(\nu_0^T BB^T \nabla\bar{u}) \,d\widetilde{S}_0 \,dt-2TE(0)\right| &= \frac{1}{2T\widetilde{E}(0)} \left| \int_{\widetilde{\Omega}} iYu\bar{u}|_0^T \,dy  \right| \\
    &\leq \frac{C}{T},
\end{align*}
and thus we obtain the asymptotic 
\begin{equation}\label{ssEQ}
    \int_0^T \int_{F_0} (n^{-1/2}\partial_{\nu_0}u)(\nu_0^T BB^T \nabla\bar{u}) \,d\widetilde{S}_0 \,dt=2T\widetilde{E}(0)\left( 1+\mathcal{O}\left( \frac{1}{T}\right) \right).
\end{equation}
Now we must transform back to the original simplex $\Omega$. We start with the right side of (\ref{ssEQ}). As the Jacobian of a matrix change of variables $x=Ay$ is $\mbox{det}(A)$, and the volume of the standard simplex is $1/n!$, we see that $ \mbox{det}(A)=n!\mbox{Vol}(\Omega). $ Thus
\begin{align*}
    2T\widetilde{E}(0)\left( 1+\mathcal{O}\left( \frac{1}{T}\right) \right)&=2T\left( 1+\mathcal{O}\left( \frac{1}{T}\right) \right) \int_{\widetilde{\Omega}} |B^T\nabla u_0|^2 \,dy \\
    &= \frac{2T}{\mbox{det}(A)}\left( 1+\mathcal{O}\left( \frac{1}{T}\right) \right) \int_{\Omega} |\nabla_x u_0|^2 \,dx \\
    &=\frac{2T}{n!\mbox{Vol}_n(\Omega)}E(0)\left( 1+\mathcal{O}\left( \frac{1}{T}\right) \right). \numberthis \label{eq1}
\end{align*}
We now work to transform the left side of (\ref{ssEQ}). We change variables from the surface measure back to the rectangular coordinates by writing $y_n$ as a graph over the other coordinates. As $F_0=\{ y_n=1-y_1-\cdots-y_{n-1} \}$, letting $dy'=dy_1\cdots dy_{n-1}$, we see that 
\begin{align*}
    d\widetilde{S}_0
    &=((dy')^2+(dy_2\cdots dy_n)^2+\cdots+(dy_1\cdots dy_{n-2} dy_n)^2)^{1/2}   \\
    &=\left(1+\left(\frac{dy_n}{dy_1}\right)^2+\cdots+\left(\frac{dy_n}{dy_{n-1}}\right)^2\right)^{1/2}dy' \numberthis \label{meas}\\
    &=(1^2+(-1)^2+\cdots+(-1)^2)^{1/2} dy' \\
    &=n^{1/2} dy'.
\end{align*}
Then we change variables from the rectangular coordinates to the surface measure on the original simplex $\Omega$. 
%For $j=1,\dots,n$, as $y_j=0$ on $F_j$, changing the variables on this face induces the $(n-1)$-dimensional volume of the $(n-1)$-dimensional parallelepiped spanned by $p_i$ for $i\neq j$. We call this parallelepiped $P_j$, and the simplex spanned by $p_i$ for $i\neq j$ is precisely the face $G_j$. 
Changing variables on $F_0$ induces the $(n-1)$-dimensional volume of the $(n-1)$-dimensional parallelepiped spanned by $p_1,p_2-p_1,\dots,p_n-p_1$, which we call $P_0$, noting that the simplex spanned by these vectors is precisely the face $G_0$. Thus we see that 
\begin{equation*}
    \mbox{Vol}_{n-1}(P_0)=(n-1)!\mbox{Vol}_{n-1}(G_0). 
\end{equation*}

Now we transform the integrand of (\ref{ssEQ}) back to the standard simplex. 
%For $j=1,\dots,n$ we have 
%\begin{equation*}
    %\nabla_y u|_{F_j} =e_j\partial_{y_j}u|_{F_j}
    %=\nu_j\partial_{\nu_j}u|_{F_j}. 
%\end{equation*}
On $F_0$, by (\ref{normal}), 
\begin{equation*}
    \nabla_y u|_{F_0} =n^{-1/2}(1,\dots,1)\partial_{\nu_0}u|_{F_0}
    =\nu_0\partial_{\nu_0}u|_{F_0}.
\end{equation*}
%Thus the following computations will apply for all $j=0,1,\dots,n$. 
Recall that $\omega_0$ is the normal to the face $G_0$ on $\Omega$. Then
\begin{align*}
    \partial_{\omega_0}u|_{G_0} &= \omega_0^T\nabla_x u|_{G_0} \\
    &=\omega_0^TB^T\nabla_y u|_{F_0} \\
    &=(\omega_0^TB^T\nu_0)\partial_{\nu_0}u|_{F_0}, 
\end{align*}
and thus 
\begin{equation}\label{int1}
\partial_{\nu_0}u|_{F_0}=\frac{\partial_{\omega_0}u|_{G_0}}{(\omega_0^TB^T\nu_0)}.
\end{equation}
As the tangential derivative vanishes, $\nu_0^T B \nabla_x\bar{u}|_{G_0}$ vanishes except for the projection onto $\omega_0$. Thus we project $\nu_0^TB$ onto $\omega_0^T$, and get that 
\begin{equation*}
    \mbox{proj}_{(\nu_0^TB)}\omega_0^T=\frac{\omega_0^T\cdot(\nu_0^TB)}{\omega_0^T\cdot\omega_0^T}\omega_0^T=(\omega_0^T(\nu_0^TB)^T)\omega_0^T=(\omega_0^TB^T\nu_0)\omega_0^T,
\end{equation*}
as $\omega_0^T$ is a unit vector. Therefore, 
\begin{align*}
    \nu_0^T BB^T \nabla_y\bar{u}|_{F_0}
    &=\nu_0^T B \nabla_x\bar{u}|_{G_0} \\
    &=(\omega_0^TB^T\nu_0)\omega_0^T\nabla_x\bar{u}|_{G_0} \numberthis \label{int2} \\ 
    &= (\omega_0^TB^T\nu_0) \partial_{\omega_0} \bar{u}|_{G_0}.
\end{align*}
Thus by our change of variables in (\ref{meas}) and our calculations in (\ref{int1}) and (\ref{int2}), the left side of (\ref{ssEQ}) now reads
\begin{flalign*}
   & \int_0^T \int_{F_0} (n^{-1/2}\partial_{\nu_0}u)(\nu_0^T BB^T \nabla\bar{u}) \,d\widetilde{S}_0 \,dt&
\end{flalign*}
\begin{align*}    
    &=\int_0^T \int_{F_0} (n^{-1/2}\partial_{\nu_0}u)(\nu_0^T BB^T \nabla\bar{u})n^{1/2} \,dy'  \,dt \\
    &=\frac{1}{\mbox{Vol}_{n-1}(P_0)}\int_0^T \int_{G_0} \left(\frac{\partial_{\omega_0}u}{(\omega_0^TB^T\nu_0)}\right)((\omega_0^TB^T\nu_0) \partial_{\omega_0} \bar{u}) \,dy' \,dt \\
    &=\frac{1}{(n-1)!\mbox{Vol}_{n-1}(G_0)}\int_0^T \int_{G_0} |\partial_{\omega_0}u|^2 \,dy' \,dt. \numberthis \label{eq2}
\end{align*}
Thus equating (\ref{eq1}) and (\ref{eq2}), we obtain our desired result on $F_0$: 
\begin{align*}
    \int_0^T \int_{G_0} |\partial_{\omega_0}u|^2 \,dy' \,dt
    &=\frac{2T(n-1)!\mbox{Vol}_{n-1}(G_0)}{n!\mbox{Vol}_n(\Omega)}E(0)\left( 1+\mathcal{O}\left( \frac{1}{T}\right) \right) \\
    &=\frac{2T\mbox{Vol}_{n-1}(G_0)}{n\mbox{Vol}_n(\Omega)}E(0)\left( 1+\mathcal{O}\left( \frac{1}{T}\right) \right).
\end{align*}

%\printbibliography
\bibliography{schrodinger}{}
\bibliographystyle{alpha}

\end{document}